\documentclass[12pt]{article}
\usepackage[a4paper]{geometry}
\usepackage{latexsym,amssymb,upref,amsmath,amsthm, amsfonts,authblk}
\usepackage{fullpage}
\newtheorem{thm}{Theorem} 
\newtheorem{lemma}[thm]{Lemma}

\newtheorem{claim}[thm]{Claim}



\newcommand{\h}{\mathcal{H}}
\newcommand{\R}{\mathcal{R}}

\newcommand{\abs}[1]{\left\lvert{#1}\right\rvert}

\title{
	On the weight of Berge-$F$-free hypergraphs
}

\begin{document}
	\author{
		Sean English \thanks{Ryerson University, Toronto, Canada. E-mail: \texttt{sean.english@ryerson.ca}} 
		\qquad
		D\'aniel Gerbner\thanks{Alfr\'ed R\'enyi Institute of Mathematics, Hungarian Academy of Sciences. E-mail: \texttt{gerbner@renyi.hu}}
		\qquad
		Abhishek Methuku\thanks{\'Ecole Polytechnique F\'ed\'erale de Lausanne, Switzerland. E-mail: \texttt{abhishekmethuku@gmail.com}}
		\qquad \\
		Cory Palmer\thanks{University of Montana, Missoula, Montana 59812, USA. E-mail: \texttt{cory.palmer@umontana.edu} }}
	
	\date{
		\today}
	\maketitle
	
	\begin{abstract}
		For a graph $F$, we say a hypergraph is a Berge-$F$ if it can be obtained from $F$ by replacing each edge of $F$ with a hyperedge containing it. A hypergraph is Berge-$F$-free if it does not contain a subhypergraph that is a Berge-$F$. The weight of a non-uniform hypergraph $\h$ is the quantity $\sum_{h \in E(\h)} \abs{h}$.
		
		Suppose $\h$ is a Berge-$F$-free hypergraph on $n$ vertices. In this short note, we prove that as long as every edge of $\h$ has size at least the Ramsey number of $F$ and at most $o(n)$, the weight of $\h$ is $o(n^2)$. This result is best possible in some sense. Along the way, we study other weight functions, and strengthen results of Gerbner and Palmer; and Gr\'osz, Methuku and Tompkins.
	\end{abstract}

\section{Introduction}

Generalizing the notion of hypergraph cycles due to Berge, the authors Gerbner and Palmer \cite{Gerbner_Palmer} introduced  so-called \textit{Berge hypergraphs}. Given a graph $F$, we say that a hypergraph $\h$ is Berge-$F$ if there is a bijection $f:E(F)\rightarrow E(\h)$ such that for every $e\in E(F)$ we have $e\subseteq f(e)$. Equivalently, $\h$ is Berge-$F$ if we can embed a distinct graph edge into each hyperedge of $\h$ to obtain a copy of $F$. Note that for a fixed $F$ there are many different hypergraphs that are Berge-$F$, and a fixed hypergraph $\h$ can be Berge-$F$ for many different graphs $F$. 

We say that a hypergraph is Berge-$F$-free if it does not contain a subhypergraph that is Berge-$F$. There are several results concerning the largest size of Berge-$F$-free hypergraphs, see e.g. \cite{anssal,FuLa,GMP,gmv,Gerbner_Palmer, G_P2,GMT,Gyori_triangle,Gyori_Lemons,GyLe4,GyLe,GyKaLe,Palmer}. For a short survey of extremal results for Berge hypergraphs see Subsection 5.2.2 in \cite{gp}.

Most of these results deal with the uniform case, but some also examine non-uniform hypergraphs. Note that replacing a hyperedge with a larger hyperedge containing it never removes a copy of Berge-$F$, but may add a copy. Thus, to build a Berge-$F$-free hypergraph that maximizes the number of hyperedges, one picks small hyperedges. To make large hyperedges more attractive, one can assign a weight to each hyperedge that increases with the size of the hyperedge. 

Gy\H ori \cite{Gyori_triangle} proved that if $\h$ is a Berge-triangle-free hypergraph, then $\sum_{h\in E(\h)}(\abs{h}-2)\le n^2/8$ if $n$ is large enough. Note that this result is about a multi-hypergraph $\h$, thus $\sum_{h\in E(\h)}\abs{h}$ can be arbitrarily large by taking a hyperedge of size $2$ an arbitrary number of times. In \cite{GyLe4}, the authors showed that for a Berge-$C_4$-free multi-hypergraph $\h$ we have $\sum_{h\in E(\h)}(\abs{h}-3)\le 12\sqrt{2}n^{3/2}+O(n)$ and they gave a construction of a Berge-$C_4$-free multi-hypergraph with approximately $n^{3/2}/8$ hyperedges. The upper bound was improved by Gerbner and Palmer \cite{Gerbner_Palmer} to $\sqrt{6}n^{3/2}/2$, while the lower bound was improved to $(1+o(1))n^{3/2}/({3\sqrt{3}})$.
For arbitrary cycles, Gy\H ori and Lemons \cite{GyLe} proved that if $\h$ is either a Berge-$C_{2k}$-free or Berge-$C_{2k+1}$-free hypergraph on $n$ vertices and every hyperedge in $\h$ has size at least $4k^2$, then $\sum_{h\in E(\h)}\abs{h}=O(n^{1+1/k})$.

Gerbner and Palmer \cite{Gerbner_Palmer} proved the following general result about Berge-$F$-free hypergraphs.



\begin{thm}[Gerbner and Palmer \cite{Gerbner_Palmer}]
	\label{gp_size}
Let $F$ be a graph and let $\h$ be a Berge-$F$-free hypergraph on $n$ vertices. If every hyperedge in $\h$ has size at least $|V(F)|$, then $\sum_{h \in E(\h)} \abs{h} = O(n^2)$.
\end{thm}

We strengthen Theorem \ref{gp_size} in Theorem \ref{higheruniformity} by showing that the statement still holds if one replaces $\abs{h}$ with $\abs{h}^2$ in the above sum; moreover, our proof is much simpler compared to the proof of Gerbner and Palmer in \cite{Gerbner_Palmer}. For uniform hypergraphs, the above theorem states that for any graph $F$ and Berge-$F$-free $r$-uniform $n$-vertex hypergraph $\h$ we have $\abs{E(\h)}=O(n^2)$ provided $r$ is large enough.
Gr\'osz, Methuku and Tompkins showed that, in fact, $\abs{E(\h)}=o(n^2)$ for large enough $r$. This is stated more precisely in the following theorem. Given two graphs $F$ and $G$, let $R(F,G)$ denote the $2$-color Ramsey number of $F$ and $G$. If $e\in E(F)$, then we write $F\setminus e$ for the graph with $V(F\setminus e)=V(F)$ and $E(F\setminus e)=E(F)\setminus\{e\}$. 

\begin{thm}[Gr\'osz, Methuku and Tompkins \cite{GMT}]
	\label{gmt_threshold}
Let $F$ be a fixed graph and $e \in E(F)$. Let $\h$ be an $r$-uniform Berge-$F$-free hypergraph. If $r \ge R(F,F\setminus e)$, then $\abs{E(\h)} = o(n^2)$. 
\end{thm}
We improve this theorem in Theorem \ref{loweruniformity}. 
Let us return to non-uniform hypergraphs. So far, we have only added up the sizes of the hyperedges. Here we will change the weight function and consider first $\sum_{h \in E(\h)} \abs{h}^2$.

\begin{thm} \label{higheruniformity}
Let $F$ be a fixed graph. Let $\h$ be a Berge-$F$-free hypergraph on $n$ vertices such that every edge of $\h$ has size at least $\abs{V(F)}$. Then 
\[
\sum_{h \in E(\h)} \abs{h}^2 = O(n^2).
\]

\end{thm}

Furthermore, this result is trivially sharp as can be seen by considering any hypergraph with at least one edge of size $\Omega(n)$. Interestingly, the next theorem shows that either small or large edges are necessary for such a weighted sum to reach this upper bound.

\begin{thm}
	\label{loweruniformity}
	Let $F$ be a fixed graph and let $e\in E(F)$. Let $\h$ be a Berge-$F$-free hypergraph on $n$ vertices such that every edge of $\h$ has size at least $R(F, F \setminus e)$ and at most $o(n)$. Then 
	\[ 
	\sum_{h \in E(\h)} \abs{h}^2 = o(n^2).
	\]
\end{thm}

Combining Theorem \ref{higheruniformity} and Theorem \ref{loweruniformity}, we can show the sum of the sizes of the edges (i.e., the weight) of a Berge-$F$-free hypergraph is $o(n^2)$ provided all the hyperedges are large enough, presenting another improvement of Theorem \ref{gp_size} and Theorem \ref{gmt_threshold}. In fact, this follows from a much more general theorem (which is presented below) by setting $w(m) = m$.


\begin{thm}
\label{main}
Let $F$ be a fixed graph and let $e\in E(F)$. Let $\h$ be a Berge-$F$-free hypergraph on $n$ vertices such that every edge of $\h$ has size at least $R(F, F \setminus e)$. If $w:\mathbb{Z}_+\rightarrow \mathbb{Z}_+$ is any weight function such that $w(m)=o(m^2)$, then 
\[
\sum_{h \in E(\h)} w(\abs{h}) = o(n^2).
\]

\end{thm}

Before we prove our results, we will comment on some of the specific conditions in Theorem \ref{main}.

Theorem \ref{main} is best possible in the sense that one cannot take a larger weight function. Indeed if $w(m) = \Omega(m^2)$, then considering a single hyperedge of size $n$ shows that the conclusion of Theorem \ref{main} cannot hold. More precisely, this gives a Berge-$F$-free hypergraph $\h$ with $\sum_{h\in \h} w(|h|) \ge w(n)$. On the other hand, for many weight functions with $w(m)=\Omega(m^2)$, the bound $O(w(n))$ is an upper bound on the weight of Berge-$F$-free hypergraphs: Indeed, if the function $\max_{1\leq i\leq n} w(i)/i^2=O\left( w(n)/n^2 \right)$ (which is achieved e.g. if $w(m)/m^2$ is eventually non-decreasing in $m$), then using Theorem \ref{higheruniformity} we have 
\[
\sum_{h \in E(\h)} w(|h|)=\sum_{h \in E(\h)} \frac{w(|h|)}{|h|^2}|h|^2 =O \left(\frac{w(n)}{n^2}\right)\sum_{h \in E(\h)} |h|^2=O(w(n)).
\]

Note that in Theorem \ref{main}, the smallest possible size of edges allowed in $\h$ must grow with the forbidden graph $F$: Indeed, let $r$ be an integer and assume $r\mid n$. Let a vertex set on $n$ vertices be partitioned into $n/r$ singletons and $n/r$ sets of size $r-1$. Let $\h$ be the $r$-uniform hypergraph consisting of all the edges that contain one singleton and one $(r-1)$-set. Then it is easily verified that $\h$ is an $r$-uniform Berge-$K_r$-free hypergraph, but $\sum_{h\in E(\h)}|h|=n^2/r$. In fact, it was shown by Gr\'osz, Methuku and Tompkins \cite{GMT} that there are $(\omega(F)-1)^2$-uniform Berge-$F$-free hypergraphs with $\Omega(n^2)$ edges, where $\omega(F)$ denotes the clique number of $F$. It is an interesting open problem to determine the smallest uniformity when $\Omega(n^2)$ drops to $o(n^2)$.

It is also worth noting that the bound $o(n^2)$ in Theorem \ref{main} is close to being best possible: Erd\H os, Frankl and R\"odl \cite{Erdos_Frankl_Rodl} constructed $r$-uniform hypergraphs with more than $n^{2-\varepsilon}$ hyperedges for any $\varepsilon$, and with the property that there are no 3 hyperedges on $3(r-1)$ vertices. Observe that a Berge-triangle is on at most $3(r-1)$ vertices, hence those hypergraphs are also Berge-triangle-free. 

\vspace{2mm}
 
\textbf{Notation.} In the rest of the paper, we use the following notation. For a set $S$ of vertices, let $\Gamma(S)$ denote the graph whose edge-set is the set of all the pairs contained in $S$. For a hypergraph $\h$, its 2-shadow is the graph whose edge-set is $\Gamma(\h) := \cup_{h \in \h} \Gamma(\{h\})$, i.e. all the edges contained in at least one hyperedge of $\h$. 

\section{Proofs}
We will use the following lemma in our proofs.
\begin{lemma}
	\label{ErdStoneSim}
Let $F$ be a non-empty graph. There exists a constant $\beta = \beta(F) > 0$ such that for any $n \ge |V(F)|$, the  maximum number of edges in an $F$-free graph on $n$ vertices is at most  $(1 - \beta) \binom{n}{2}$. 
\end{lemma}

\begin{proof}
Let us fix a real number $\alpha$ such that $1>\alpha>\frac{1}{\chi(F)-1}$ if $\chi(F)\ge 3$ and $0<\alpha<1$ if $\chi(F)=2$. According to the Erd\H os-Stone-Simonovits theorem there exists an $n_0\ge |V(F)|$ such that any $F$-free graph on $n\ge n_0$ vertices contains at most $(1-\alpha)\binom{n}{2}$ edges. On the other hand, if $|V(F)| \le n < n_0$, then obviously an $F$-free graph contains at most $\binom{n}{2} - 1 \le \left(1 - \frac{1}{\binom{n_0}{2}}\right) \binom{n}{2}$ edges. Therefore, letting $\beta:=\min\left\{\alpha,\frac{1}{\binom{n_0}{2}}\right\}$ proves the lemma.
\end{proof}

\subsection{Proof of Theorem \ref{higheruniformity}}
We will say an edge in $\Gamma(\h)$ is \emph{blue} if it is contained in at most $\abs{E(F)}-1$ hyperedges of $\h$. 

\begin{claim}
\label{blueineveryF}
Every copy of $F$ in $\Gamma(\h)$ contains a blue edge.
\end{claim}
\begin{proof}
Consider a copy of $F$ in $\Gamma(\h)$. If there is no blue edge in $F$ then every edge of $F$ is contained in at least $\abs{E(F)}$ hyperedges of $\h$ by definition, so one can greedily choose different hyperedges representing the edges of $F$. Thus we have a Berge-$F$ in $\h$, a contradiction.
\end{proof}

The following claim bounds the number of blue edges in a hyperedge of $\h$ from below.

\begin{claim}
\label{manyblueedges}
Let $h \in E(\h)$ be a hyperedge. Then there exists a constant $\beta = \beta(F)>0$, such that the number of blue edges in $\Gamma(h)$ is at least $\beta\binom{\abs{h}}{2}$.  
\end{claim}
\begin{proof}
	The graph $\Gamma(h)$ is a clique on $|h|\geq |V(F)|$ vertices, and by Claim \ref{blueineveryF}, the set of blue edges in $\Gamma(h)$ form a graph, the complement of which is $F$-free. Thus, Lemma \ref{ErdStoneSim} guarantees a constant $\beta=\beta(F)>0$ such that there are at least $\beta\binom{\abs{h}}{2}$ blue edges in $\Gamma(h)$.



\end{proof}

Using Claim \ref{manyblueedges}, we have 

\begin{equation}\label{equation sum h choose 2}
\sum_{h \in E(\h)}  \beta\binom{\abs{h}}{2} \le \sum_{h \in E(\h)} \# \{\text{Blue edges in $\Gamma(h)$}\}.
\end{equation}

On the other hand, we have 
\begin{equation}\label{equation On^2}
\sum_{h \in E(\h)} \# \{\text{Blue edges in $\Gamma(h)$}\} \le \# \{\text{Blue edges in $\Gamma(\h)$}\} \cdot (\abs{E(F)}-1) = O(n^2).
\end{equation}

Indeed each blue edge is counted at most $\abs{E(F)}-1$ times in the summation as it is contained in at most $\abs{E(F)}-1$ hyperedges of $\h$. Then combining equations \eqref{equation sum h choose 2} and \eqref{equation On^2}, we have $\sum_{h \in E(\h)}  \beta\binom{\abs{h}}{2}=O(n^2)$ for constant $\beta$, which implies that $\sum_{h \in E(\h)}  \abs{h}^2 \le \sum_{h \in E(\h)} 4\binom{\abs{h}}{2} =O(n^2)$, completing the proof.

\subsection{Proof of Theorem \ref{loweruniformity}}

If $F$ has one or fewer edges, the statement is trivial, so we will assume $F$ has at least two edges throughout the rest of the proof. Here we follow an argument similar to Gr\'osz, Methuku and Tompkins \cite{GMT} but with some important changes. We wish to apply the graph removal lemma to the 2-shadow of a hypergraph $\h$. To this end, we prove the following claim.

\begin{claim}
\label{numberbound}
The number of copies of $F$ in $\Gamma(\h)$ is $o(n^{\abs{V(F)}})$.
\end{claim}

\begin{proof}
Any copy of $F$ in $\Gamma(\h)$ has at least two edges (and therefore at least three vertices) in some hyperedge of $\h$, otherwise the hyperedges containing the edges of $F$ would form a Berge-$F$. Thus we have the following upper bound: 
\begin{displaymath}
\# \{\text{$F$-copies in $\Gamma(\h)$}\} \le \sum_{h \in \h} \binom{|h|}{3} n^{\abs{V(F)}-3} \binom{\binom{\abs{V(F)}}{2}}{\abs{E(F)}}\le n^{\abs{V(F)}-3}\binom{\binom{\abs{V(F)}}{2}}{\abs{E(F)}}\sum_{h \in \h} |h|^3.
\end{displaymath}
Indeed, there are $\binom{|h|}{3}$ ways to select three vertices from a hyperedge $h \in \h$ and there are at most $n^{\abs{V(F)}-3}$ ways to select the remaining $\abs{V(F)}-3$ vertices to form a set of $\abs{V(F)}$ vertices. The number of copies of $F$ in this set is bounded by $\binom{\binom{\abs{V(F)}}{2}}{\abs{E(F)}}$.

By our assumption, $|h| = o(n)$, and by Theorem~\ref{higheruniformity}, we have $\sum_{h \in \h} \abs{h}^2=O(n^2)$, so $\sum_{h \in \h} |h|^3 = \sum_{h \in \h} |h| \cdot |h|^2 = o(n) \sum_{h \in \h} |h|^2 = o(n^3)$. Therefore, the number of copies of $F$  in $\Gamma(\h)$ is $o(n^{\abs{V(F)}})$, proving the claim.
\end{proof}

By Claim~\ref{numberbound} and the graph removal lemma, there is a set  $\R$ of $o(n^2)$ edges in $\Gamma(\h)$ such that every copy of $F$ in the 2-shadow of $\h$ contains an edge of $\R$. We will call an edge in the 2-shadow of $\h$ \emph{special} if it is contained in $\R$ and is contained in at most $\abs{E(F)}-1$ hyperedges. Note that the special edges here play a similar but slightly different role than the blue edges in the proof of Theorem \ref{higheruniformity}. Let $\R_s$ be the set of all the special edges. Of course, $\R_s \subseteq \R$.

Recall that $e\in E(F)$, and $R(F,F\setminus e)$ denotes the Ramsey number of $F$ versus $F \setminus e$.

\begin{claim}\label{counting}
	\label{findingonespecialedge}
	Let $h \in E(\h)$ be an arbitrary hyperedge. Then any subset  $S \subseteq h$ of size $R(F,F\setminus e)$ contains a special edge (i.e., $\Gamma(S) \cap \R_s \not = \emptyset$).
\end{claim}


\begin{proof}
Assume by contradiction that there is a set $S \subseteq h$ of size $R(F, F \setminus e)$ which contains no special edge. In other words, every edge of $\R$ contained in $S$ is in at least $\abs{E(F)}$ hyperedges. By the definition of $\R$, $\Gamma(S)\setminus \R$ cannot contain a copy of $F$. Applying Ramsey's theorem with the edges of $\Gamma(S)\setminus \R$ colored with the first color and those in $\Gamma(S)\cap\R$ colored with the second, we obtain that $\Gamma(S)\cap\R$ must contain a copy of $F\setminus e$. Let $\hat e$ be an edge contained in $S$ whose addition would complete this copy of $F$. The other edges of this copy of $F$ are each contained in at least $\abs{E(F)}$ hyperedges of $\h$. Thus we can select greedily $\abs{E(F)}$ different hyperedges of $\h$ to represent the edges in this copy of $F$: $h$ itself for $\hat e$, and $\abs{E(F)}-1$ other hyperedges for the rest of the edges of $F$. These hyperedges form a Berge-$F$ in $\h$, a contradiction.
\end{proof}

Now we provide a lower bound on the number of special edges contained in a hyperedge of $\h$.

\begin{claim}
\label{findingmanyspecialedges}
Let $h \in \h$ be a hyperedge. Then there is a constant $\gamma = \gamma(F)$ such that $$\abs{\Gamma(h) \cap \R_s} \ge \gamma\binom{\abs{h}}{2}.$$
\end{claim}
\begin{proof}
Claim \ref{findingonespecialedge} implies that $\Gamma(h) \setminus \R_s$ does not contain a complete graph on $R(F, F \setminus e)$ vertices. So by Lemma \ref{ErdStoneSim}, $\Gamma(h) \setminus \R_s$ contains at most $(1-\gamma) \binom{\abs{h}}{2}$ edges for some constant $\gamma = \gamma(F)$. So $\Gamma(h) \cap \R_s$ contains at least $\gamma\binom{\abs{h}}{2}$ edges, as desired.
\end{proof}

Now since $\R_s \subseteq \R$, we have $\abs{\R_s} = o(n^2)$. This fact together with Claim \ref{findingmanyspecialedges} implies the following.

$$\sum_{h \in \h} \gamma\binom{\abs{h}}{2} \le \sum_{h \in \h} \abs{\Gamma(h) \cap \R_s} \le \abs{\R_s} (\abs{E(F)}-1) = o(n^2).$$

Indeed, the sum $\sum_{h \in \h} \abs{\Gamma(\{h\}) \cap \R_s}$ counts each edge of $\R_s$ at most $\abs{E(F)}-1$ times.






\subsection{Proof of Theorem \ref{main}}


Since $w(m) = o(m^2)$ and $w$ is defined only on $\mathbb{Z}_+$, there are only finitely many values of $m$ such that $w(m)>m^2$, and thus $w(m)=O(m^2)$. Let $C$ be a constant such that $w(m) \le C m^2$ for all $m\in\mathbb{Z}_+$. Theorem \ref{loweruniformity} implies that

\begin{equation}
\label{eq:eq1}
\sum_{h \in E(\h): |h| \le n^{1/2}} w(|h|) \le \sum_{h \in E(\h): |h| \le n^{1/2}} C |h|^2  = o(n^2),
\end{equation} simply because $n^{1/2} = o(n)$.  Now since $w(m) = o(m^2)$, Theorem \ref{higheruniformity} implies that
\begin{equation}
\label{eq:eq2}
\sum_{h \in E(\h): |h| > n^{1/2}} w(|h|)  = \sum_{h \in E(\h): |h| > n^{1/2}} o(|h|^2) = o\left(\sum_{h \in E(\h): |h| > n^{1/2}} |h|^2\right) = o(n^2).
\end{equation}

So adding up \eqref{eq:eq1} and \eqref{eq:eq2}, the proof is complete.


\subsection*{Acknowledgements}
We thank J\'ozsef Balogh for suggesting the line of investigation carried out in this paper. 

\vspace{1mm}

The research of Gerbner was supported by the J\'anos Bolyai Research Fellowship of the Hungarian Academy of Sciences and by the National Research, Development and Innovation Office -- NKFIH, grant SNN 116095, grant K 116769 and grant KH 130371.

The research of Methuku was partially supported by the National Research, Development and Innovation Office NKFIH grant K116769.

The research of Palmer was partially supported by University of Montana UGP Grant \#M25460.


\begin{thebibliography}{99}


\bibitem{anssal} R. Anstee and S. Salazar. Forbidden Berge hypergraphs. Electronic
Journal of Combinatorics, 24(1), 2017. P1.59.



\bibitem{Erdos_Frankl_Rodl}
P. Erd\H os, P. Frankl and V. R\"odl. The asymptotic number of graphs not
containing a fixed subgraph and a problem for hypergraphs having no
exponent. Graphs and Combinatorics 2(1), 113--121, (1986). doi:10.1007/BF01788085

\bibitem{FuLa} Z. F\"uredi and L. \"Ozkahya. On 3-uniform hypergraphs without a cycle of a given length.
Discrete Applied Mathematics, 216, 582--588, (2017).

\bibitem{GMP} D. Gerbner, A. Methuku and C. Palmer. General lemmas for Berge-Tur\'an hypergraph problems, arXiv preprint arXiv:1808.10842 (2018).

\bibitem{gmv} D. Gerbner, A. Methuku and M. Vizer. Asymptotics for the Tur\'an number of Berge-$K_{2,t}$, Journal of Combinatorial Theory, Series B, to appear

\bibitem{Gerbner_Palmer}
D. Gerbner and C. Palmer. Extremal results for Berge-hypergraphs. SIAM Journal on Discrete Mathematics, 31(4): 2314--2327 (2015). doi:10.1137/16M1066191

\bibitem{G_P2} D. Gerbner and C. Palmer. Counting copies of a fixed subgraph in $ F $-free graphs. {\it arXiv preprint} arXiv:1805.07520 (2018).

\bibitem{gp} D. Gerbner and B. Patk\'os. Extremal Finite Set Theory,
1st Edition, CRC Press, 2018.

\bibitem{GMT}
D. Gr\'osz, A. Methuku and C. Tompkins. Uniformity thresholds for the asymptotic size of extremal Berge-$F$-free hypergraphs. arXiv preprint arXiv:1803.01953 (2018).

\bibitem{Gyori_triangle} E. Gy\H ori. Triangle-Free Hypergraphs. Combinatorics, Probability and Computing, 15 (1-2): 185--191 (2006). doi:10.1017/S0963548305007108

\bibitem{GyKaLe}
E. Gy\H ori, G.Y. Katona and N. Lemons.
Hypergraph extensions of the {E}rd{\H{o}}s-{G}allai theorem.
Electronic Notes in Discrete Mathematics, 36: 655--662 (2010).

\bibitem{Gyori_Lemons} E. Gy\H ori and N. Lemons. 3-uniform hypergraphs avoiding a given odd cycle. Combinatorica, 32: 187--203 (2012). doi:10.1007/s00493-012-2584-4

\bibitem{GyLe4} E. Gy\H ori and N. Lemons. Hypergraphs with no cycle of length 4, Discrete Math., 312 (2012),
541 pp. 1518--1520.

\bibitem{GyLe} E. Gy\H ori and N. Lemons. Hypergraphs with no cycle of a given length. Combin.
Probab. Comput., 21(1-2):193--201, 2012.


\bibitem{Palmer}
C. Palmer,  M. Tait, C. Timmons and A.Z. Wagner. Tur\'an numbers for Berge-hypergraphs and related extremal problems. arXiv preprint arXiv:1706.04249 (2017).


\end{thebibliography}
\end{document}